\documentclass[12pt]{article}

\usepackage{graphicx}
\usepackage{microtype}
\usepackage{amsmath}
\usepackage{amsthm}
\usepackage{amssymb}
\usepackage{hyperref}
\usepackage{geometry}
\usepackage{epstopdf}
\usepackage{geometry}
\usepackage{standalone}
\usepackage{tikz}

\newtheorem{theorem}{Theorem}[section]
\newtheorem{lemma}[theorem]{Lemma}
\newtheorem{proposition}[theorem]{Proposition}

\theoremstyle{definition}
\newtheorem{definition}[theorem]{Definition}

\newcommand{\seqnum}[1]{\href{https://oeis.org/#1}{\underline{#1}}}

\usetikzlibrary{decorations.pathmorphing}
\usetikzlibrary{arrows}
\usetikzlibrary{chains,fit,shapes}

\tikzset{snake it/.style={decorate, decoration=snake}}

\tikzstyle{doto} = [dash pattern= on 0.7cm off 0.1cm on \pgflinewidth off 4\pgflinewidth on \pgflinewidth off 4\pgflinewidth on \pgflinewidth off 0.1cm on 0.7cm]

\DeclareMathOperator{\young}{young}
\DeclareMathOperator{\eld}{eld}
\DeclareMathOperator{\Sel}{Sel}
\DeclareMathOperator{\unl}{unl}
\DeclareMathOperator{\impe}{impe}
\DeclareMathOperator{\impp}{impp}
\DeclareMathOperator{\lead}{lead}
\DeclareMathOperator{\bap}{gap}
\DeclareMathOperator{\repr}{R}

\everymath{\displaystyle}

\makeatletter
\def\input@path{{figures/}}
\makeatother

\title{Arboretum for a generalization of Ramanujan polynomials}

\author{Lucas Randazzo\thanks{\href{mailto:lucas.randazzo@u-pem.fr}{lucas.randazzo@u-pem.fr}.}}

\begin{document}

\maketitle

\begin{abstract}
In this paper, we expand on the work of Guo and Zeng from 2007 on a generalization of the Ramanujan polynomials and planar trees. We manage to find combinatorial interpretations of this family of polynomials in terms of Greg trees, Cayley trees, and planar trees by constructing bijections that preserve relevant tree statistics. 
\end{abstract}

\section{Introduction}
Study of the Ramanujan's sequence (\ref{eq:ramseq}) has led to findings regarding refinements to Cayley's formula (see \cite{berndt1985ramanujan,guo2007generalization,zeng1999ramanujan}). This formula gives us that the number of \emph{rooted labelled trees} with $n$ vertices is $n^{n-1}$. As an example of refinement, we have the polynomials $C_n$ defined as follows
\[
C_1 = 1, \hspace{0.5cm} C_{n+1}(y) = \left[ n(1+y)+y^2\partial_y \right]C_n
\]
and verifying (see \cite{zeng1999ramanujan}) $C_n(1) = n^{n-1}$. We have $C_2 = y + 1$ and $C_3 = 3y^2 + 4y+2$. It has been proven by Shor \cite{shor1995new} and Dumont and Ramamonjisoa \cite{dumont1996grammaire} that the coefficient of $y^k$ in $C_n(y)$ counts rooted labelled trees with $n$ vertices and $k$ \emph{improper edges} (see \seqnum{A217922}). These polynomials originate from the polynomials $P_n$, defined as
\[
P_1 = 1, \hspace{0.5cm} P_{n+1}(x,y) = \left[ x+n+y(n+y\partial_y) \right]P_n
\]
which verifies $P_n(0,y) = C_n(y)$. The polynomials $P_n$ are called the \emph{Ramanujan's polynomials} \cite{chapoton2002operades,chen2001bijections}. Indeed, originally Ramanujan \cite{ramanujan2013notebooks} introduced the following double sequence $\psi_k(r,x)$ as follows
\begin{equation}
\label{eq:ramseq}
\sum_{k=0}^{+\infty}\frac{(x+k)^{r+k}e^{-u(x+k)}u^k}{k!} = \sum_{k=1}^{r+1}\frac{\psi_k(r,x)}{(1-u)^{r+k}}.
\end{equation}
Berndt et al. \cite{berndt1985ramanujan,berndt1983ramanujan} would then define the $P_n$ polynomials as
\[
P_n(x,y) = \sum_{k=0}^{n-1}\psi_{k+1}(n-1,x+n)y^k.
\]
Following this, Guo and Zeng study in \cite{guo2007generalization} the following polynomial sequence, credited to Chapoton, as a generalization of the previous sequence 
\begin{equation}
\label{eq:qrec}
Q_1 = 1, \hspace{0.5cm} Q_{n+1}(x,y,z,t) = [x+nz+(y+t)(n+y\partial_y)]Q_n,
\end{equation}
verifying $Q_n(x,y,0,1) = P_n(x,y)$. We have for instance $Q_2 = x +y+t+z$ and
\[
Q_3 = x^2 + 3xy+3xz+3xt+3y^2+4yz+5yt+2z^2+4zt+2t^2.
\]
They find a combinatorial interpretation of this sequence in terms of \emph{planar labelled trees}. However, as stated by Josuat-Verg\`es in \cite{josuat2015derivatives}, the Ramanujan polynomials seem to also be correlated to \emph{Greg trees}. Initially defined by Flight in \cite{flight1990many} to represent genealogical trees for manuscripts, Greg trees are trees with both labelled and unlabelled nodes, the unlabelled ones being of degree at least three. 

Our goal is to relate the polynomials $Q_n$ with \emph{Greg trees} and \emph{Cayley trees}, and give a simple expression of $Q_n$ using different tree statistics. In Section 2 we introduce the different notions and definitions needed throughout this article, in particular we will define a sequence $R_n$ that will be key to most of our proofs. In Section 3, we will express $R_n$ in terms or relevant statistics on Greg trees using a recurrence relation for $R_n$. Section 4 will give a similar result for Cayley trees, by using a bijection between Greg trees and Cayley trees. Finally in Section 5 we will show how this relates to $Q_n$ by introducing a bijection between Cayley trees and planar labelled trees.

\section{Preliminaries}
We recall that a \emph{Cayley tree} is a tree with $n$ labelled vertices on $[n]$ such that no two vertices are labelled the same. Let $\mathcal{C}_n$ be the set of Cayley trees of size $n$ rooted at $1$, we have $|\mathcal{C}_n| = n^{n-2}$. A \emph{planar tree} is a rooted labelled tree in which the children of each vertex are linearly ordered. We denote by $\mathcal{O}_{n}$ the set of planar trees rooted at $1$ and labelled on $[n]$. The following definition, coined by Flight in \cite{flight1990many}, gives us an interesting generalization of Cayley trees.
 
\begin{definition}[Flight, \cite{flight1990many}]
Let $n \geq 1$. A \emph{Greg tree} of size $n$ is a tree such that exactly $n$ of its vertices are labelled in $[n]$, and the unlabelled vertices are of degree at least $3$.
\end{definition}

Let $\mathcal{G}_n$ be the set of Greg trees of size $n$ rooted at $1$. Let $\unl(T)$ denote the number of unlabelled vertices of $T \in \mathcal{G}_n$. If $\unl(T) = 0$, then $T \in \mathcal{C}_n$. See Figure \ref{img:greg} for an example of a Greg tree.

For a labelled tree $T$, especially for Greg trees, it is convenient to introduce $\lambda_T$ the labelling function of $T$. As a convention, we choose for $\lambda_T$ to not be defined over unlabelled vertices in a Greg tree.

\begin{figure}
\center
\begin{tikzpicture}

\path
(0,0) node(a1) {1}
++(1,-1) node(b1) {7}
++(0,-1) node[circle,draw](c1) {$u$}
+(0.5,-1) node(d1) {6}
+(-0.5,-1) node(d2) {9}
(-1,-1) node(b2) {4}
+(1,-1) node(c2) {8}
+(0,-1) node(c3) {2}
++(-1,-1) node(c4) {5}
+(0,-1) node(d3) {3};
\draw (b1) -- (a1) -- (b2) -- (c2) (d1) -- (c1) -- (d2);
\draw[double] (b1) -- (c1) (c3) -- (b2) -- (c4) -- (d3);

\end{tikzpicture}
\caption{An example of a Greg tree of size 9, with one unlabelled vertex $u$. The improper edges are doubled.}
\label{img:greg}
\end{figure}
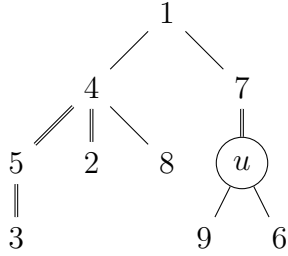

Since all the trees we consider are rooted, we will use oriented edges, with the following notation: if $(i,j)$ is an edge of $T$, then $i$ is the parent of $j$.

In an alternative proof of Cayley's formula \cite{shor1995new}, Shor introduced an interesting categorization of edges in a Cayley tree, which can be generalized to Greg trees.

\begin{definition}[Shor, \cite{shor1995new}]
Let $T \in \mathcal{C}_n \cup \mathcal{G}_n \cup \mathcal{O}_n$, and $i$ a labelled vertex of $T$. We define $\beta_T(i)$ to be the smallest label in the descendants of $i$. If $T \in \mathcal{C}_n \cup \mathcal{G}_n$, let $e=(i,j)$ be an edge of $T$. We say that $e$ is an \emph{improper edge}, or $j$ is an \emph{improper child} of $i$, if $i$ is labelled and $\lambda_T(i) > \beta_T(j)$. Otherwise, $e$ is a \emph{proper edge}, and $j$ is a \emph{proper child} of $i$. 
\end{definition}

For instance in Figure \ref{img:greg}, improper edges are represented with double lines, and we have $\beta_T(4) = 2$ and $\beta_T(7) = 6$. Naturally, $\beta_T$ is also defined on unlabelled vertices, and here we have $\beta_T(u) = 6$. Remark that an unlabelled vertex cannot be an improper parent, even though it can be an improper child.

Regarding planar trees, we need to use Guo and Zeng's generalization from \cite{guo2007generalization}.

\begin{definition}
Let $T \in \mathcal{O}_n$.
\begin{itemize}
\item Let $j$ be a vertex of $T$. $j$ is \emph{elder} if it has a brother $k$ to its right such that $\beta_T(k) < \beta_T(j)$. Otherwise, we say that $j$ is \emph{younger}. Let $\young_T(i)$ be the number of younger children of $i$ in $T$, and $\eld(T)$ be the number of elder vertices of $T$.

\item Let $e=(i,j)$ be an edge of $T$. We say that $e$ is an \emph{improper edge}, or $j$ is an improper child of $i$, if $j$ is a younger child of $i$ and $\lambda_T(i) > \beta_T(j)$. Otherwise, $e$ is a \emph{proper edge}, and $j$ is a \emph{proper child} of $i$.

\end{itemize}
\end{definition}

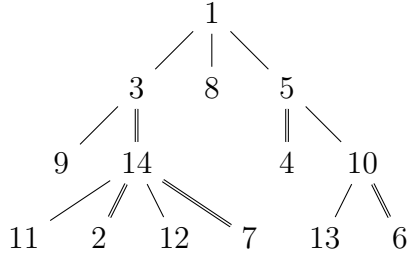
\begin{figure}
\center
\begin{tikzpicture}

\path
(0,0) node(a1) {1}
++(1,-1) node(b1) {5}
++(1,-1) node(c1) {10}
+(0.5,-1) node(d1) {6}
+(-0.5,-1) node(d2) {13}
(b1)++(0,-1) node(c2) {4}
(a1)++(0,-1) node(b2) {8}
++(-1,0) node(b3) {3}
++(0,-1) node(c3) {14}
++(1.5,-1) node(d3) {7}
++(-1,0) node(d4) {12}
++(-1,0) node(d5) {2}
++(-1,0) node(d6) {11}
(b3)++(-1,-1) node(c4) {9};

\draw (d2) -- (c1) -- (b1) -- (a1) -- (b3) -- (c4) (a1) -- (b2) (d4) -- (c3) -- (d6);
\draw[double] (d1) -- (c1) (c2) -- (b1) (d3) -- (c3) -- (d5) (c3) -- (b3) ;

\end{tikzpicture}
\caption{An example of a planar tree of size 14. The improper edges are shown doubled.}
\label{img:plan}
\end{figure}

We show an example of a planar tree, with improper edges doubled, in Figure \ref{img:plan}. In either case, let $\impe(T)$ be the number of improper edges of $T$. 
Remark that if $T$ is a planar tree with no elder child, its siblings are ordered by increasing order of $\beta_T$. In this case, the improper edges are the same as if we removed the order on the siblings of $T$.

Finally, we can extend the definition of \emph{improper} to vertices, when we want to know whether a vertex has improper outgoing edges or not.

\begin{definition}
\label{def:impp}
We say that a parent $i$ is \emph{proper} if it has no improper child. Otherwise we say that $i$ is an \emph{improper parent}.  Equivalently, a vertex $i$ is an improper parent if and only if it is labelled and $\beta_T(i) \neq \lambda_T(i)$. We denote by $\impp(T)$ the number of improper parents of $T$.
\end{definition}

We can now introduce the main theorem of Guo and Zeng in \cite{guo2007generalization}. It gives a combinatorial interpretation of $Q_n$ in terms of planar trees.

\begin{theorem}[Guo and Zeng, \cite{guo2007generalization}]
\label{th:qn}

\begin{equation*}
Q_n(x,y,z,t) = \sum_{T \in \mathcal{O}_{n+1}} x^{\young_T(1)-1}t^{\eld(T)}y^{\impe(T)}z^{n-\young_T(1)-\eld(T)-\impe(T)}.
\end{equation*}
\end{theorem}

Our goal is to find a similar equality for Greg trees and Cayley trees. Let us consider the following polynomials.

\begin{definition}[Josuat-Verg\`es, \cite{josuat2015derivatives}]
\begin{equation}
\label{eq:hrec}
H_1 = 1, \hspace{0.5cm} H_{n+1}(y) = (2n-1+(n+1)y)H_n(x) + (1+y)^2H_n'(x).
\end{equation}
\end{definition}

These polynomials count the number of Greg trees with $n$ labelled vertices (\seqnum{A048159}). Using (\ref{eq:qrec}) and (\ref{eq:hrec}), we obtain
\begin{equation}
\label{eq:hisq}
H_{n+1}(y) = Q_n(1,y+1,1,0).
\end{equation}

This suggests that there is a correlation between Greg trees and planar trees. Moreover, we have

\begin{proposition}[Josuat-Verg\`es, \cite{josuat2015derivatives}]
\label{prop:h}
\begin{equation}
\label{eq:hbe}
H_n(y) = \sum_{T \in \mathcal{G}_n}y^{\unl(T)} = \sum_{T \in \mathcal{C}_n} (1+y)^{\impe(T)}.
\end{equation}
\end{proposition}

The original paper asks about the existence of a bijective proof for (\ref{eq:hbe}). In the Section 4 we will indeed construct such a bijection. We will actually find a bijection that generalizes Proposition \ref{prop:h} for the following polynomials in 4 variables.

\begin{definition}
\label{def:rn}
\begin{equation}
\label{eq:defrn}
R_n(x,y,z,t) = Q_n(x,y+1,z,t-1)
\end{equation}
\end{definition}
$R_n$ allows us to bridge the gap between $H_n$ and $Q_n$, since we can see $R_n$ as a generalization of $H_n$, as we have $R_n(1,y,1,1)~=~H_n(y)$. In particular we have $R_{n+1}(0,y,0,1) = (2n-1)!!(y+1)^n$, and $R_n(0,y,0,0)$ gives the Ward numbers (\seqnum{A134991}). We will explain in more details why this last property is true at the end of Section 3.

Before moving on to the next section, we need to introduce one more statistic that will be useful in the next section.

\begin{definition}
Let $T$ be a Greg or Cayley tree. Let $i$ be a labelled vertex of $T$. Let $L(i) = (1=a_0,a_1,\ldots,a_k = i)$ be the only path through $T$ from the root $1$ to $i$. We define the \emph{greater ancestors path} of $i$, noted $\bap(i) = (a_p,a_{p+1},\ldots,i)$, the longest path in $T$ included in $L(i)$ and containing $i$ such that, for every labelled vertex $j \in \bap(i)$, $\lambda_T(j) \geq \lambda_T(i)$.
\end{definition}

For example, in Figure \ref{img:greg}, we have $\bap(3) = (4,5,3)$, $\bap(2) = (4,2)$ and $\bap(6)~=~(6)$.

\begin{definition}
Let $i \in T$ and $\bap(i) = (a_p,\ldots,i)$. $i$ is a \emph{leading vertex} if $\beta_T(a_p)~=~\lambda_T(i)$.
\end{definition}
Remark that $1$ and $2$ are always leading vertices. Also remark that if $i$ is a leading vertex, then for all $j \in \bap(i)$, $\beta_T(j) = \lambda_T(i)$. This also implies with Definition \ref{def:impp} that $i$ is a proper parent.

We will see that the number of leading vertices is relevant for Greg trees, but not in Cayley or planar trees. Hence we need the following Lemma.

\begin{lemma}
\label{lem:leadimp}
Let $T \in \mathcal{C}_n$. We have
\begin{equation} \lead(T) = n-\impe(T).\end{equation}
\end{lemma}
\begin{proof}
Since every edge is either proper or improper, we have that $n-1-\impe(T)$ is the number of proper edges of $T$. Hence to prove the lemma, we can find a bijection between the set of leading vertices except $1$, and the proper edges of $T$. 

Let $i$ be a leading vertex of $T$ that is not $1$.  Let $\bap(i) = (a_p,\ldots,i)$. Since $i$ is not $1$, let $a_{p-1}$ be the parent of $a_p$. Let $\phi(i) = (a_{p-1},a_p)$. 

$\phi$ is injective. Indeed, let $u$ be another leading vertex of $T$, and assume that $\phi(u) = \phi(i)$. We would have $a_p \in \bap(i) \cap \bap(u)$, hence verifying both $\beta_T(a_p) = \lambda_T(i)$ and  $\beta_T(a_p) = \lambda_T(u)$, which is impossible.

$\phi$ is surjective. Let $(a,b)$ be a proper edge of $T$. From $b$, we can go down the tree by choosing the child with the smallest value for $\beta_T$, until we find a proper parent $u$, which we can always find since leaves are. On this path, $\beta_T$ is constant, equal to $\lambda(u)$. Hence this path is a suffix of $\bap(u)$. Since $(a,b)$ is a proper edge, $a \not\in \bap(u)$, so the path from $b$ to $u$ is $\bap(u)$, and $\phi(u) = (a,b)$.

Hence, $\phi$ is bijective, which proves the lemma.
\end{proof}
\section{Greg Trees and $R_n$}
The first result we show is that $R_n$ can be seen as a generating function for Greg trees. This is a generalization of the first equality of Proposition \ref{prop:h}.
\begin{theorem}
\label{th:main1}
\begin{equation}
\label{eq:main1}
R_n(x,y,z,t) = \sum_{T \in \mathcal{G}_{n+1}} y^{\unl(T)} t^{\impp(T)} x^{\deg_T(1)-1} z^{\lead(T)-\deg_T(1)-1}.
\end{equation}
\end{theorem}

\begin{proof}
The definition of $R_n$ (\ref{eq:defrn}) gives us the following recurrence:
\begin{equation}
\label{eq:prn}
R_1 = 1, \hspace{0.5cm} R_{n+1} = [x +nz + (y+t)(n+(y+1)\partial_y)]R_n.
\end{equation}
Now we just need to prove that the right hand side in (\ref{eq:main1}) is a solution to (\ref{eq:prn}). We define a weight function over Greg trees as 
\[
\omega(T) = y^{\unl(T)} t^{\impp(T)} x^{\deg_T(1)-1} z^{\lead(T)-\deg_T(1)-1}.
\] 
The recurrence can be interpreted as the weighted sum of the different ways to add a $n+2$-labelled vertex to a Greg tree of size $n+1$. In the following, we denote by $T$ the base tree of size $n+1$, and $T'$ the result of adding a vertex labelled $n+2$ to $T$. We have eight distinct operations, noted $f_i$ for $i \in [1,8]$:
\begin{enumerate}
\item $\textstyle \sum_{T' \in f_1(T)}\omega(T') = x\omega(T)$. 

We add the new vertex labelled $n+2$ as a child of $1$. This new vertex has to be leading, since its parent is smaller that itself, and $\beta_T(n+2)~=~n~+~2$, so $\lead(T')~=~\lead(T)~+~1$. We also increase the degree of $1$, so we have $\deg_{T'}(1)~=~\deg_T(1)~+~1$. Since \[ \lead(T')-\deg_{T'}(1) = \lead(T) + 1-(\deg_T(1) +1) = \lead(T) - \deg_T(1), \] the degree of $z$ in $\omega(T')$ does not change.

\item $\textstyle \sum_{f_2(T)}\omega(T') = nz\omega(T)$. 

We add the new vertex labelled $n+2$ as a child to another labelled vertex that is not $1$, which makes $n$ distinct vertices to chose from. Like in the previous case, this new vertex is leading, but the degree of $1$ does not change, so we need to multiply by $z$ instead of $x$ in this case.

\item $\textstyle \sum_{f_3(T)}\omega(T') = y\partial_y\omega(T)$. 

We add the vertex labelled $n+2$ as a child to an existing unlabelled vertex $u$. Here we choose amongst the unlabelled vertices of $T$, which are counted by the degree of $y$ in $\omega(T)$. Multiplying by the degree of $y$ in the equation translates to deriving with respect to $y$, then multiplying by $y$ the weight function. In this case, the new vertex is not leading. Indeed, its parent $u$ is unlabelled, so it has at least two other descendants smaller than $n+2$, so $\beta_{T'}(u) = \beta_T(u) \neq n+2$. The degree of $1$ also stays unchanged.

\item $\textstyle \sum_{f_4(T)}\omega(T') = t\partial_y\omega(T)$. 

We relabel an unlabelled vertex with $n+2$. Since we remove an unlabelled vertex from $T$, we do not have $y$ as a factor here in the recursion. However, $\beta_{T'}(n+2) \neq n+2$, so $n+2$ is a new improper parent and it is not a leading vertex, so we only add a factor $t$.

\item $\textstyle \sum_{f_5(T)}\omega(T') = tn\omega(T)$. 

We add the vertex $n+2$ in the middle of an existing edge $(i,j)$, where $j$ is labelled. We have $n$ choices of edges, as $j$ can be any labelled vertex of $T$ except $1$. We have that $\beta_{T'}(n+2) = \beta_{T'}(j) = \beta_T(j) \neq n+2$, so $n+2$ is a new improper parent and not a leading vertex. 

\item $\textstyle \sum_{f_6(T)}\omega(T') = ty\partial_y\omega(T)$. 

We add the vertex $n+2$ in the middle of an existing edge $(i,j)$, where $j$ is unlabelled. The same reasoning as the previous case applies to this one, so $n+2$ is a new improper parent and not a leading vertex.

\item $\textstyle \sum_{f_7(T)}\omega(T') = yn\omega(T)$. 

We add an unlabelled vertex $u$ in the middle of an edge $(i,j)$ where $j$ is labelled, then add the vertex $n+2$ as its second child. $u$ cannot be an improper parent, since it is not labelled. Moreover, similarly to case $3$, as the highest labelled child of an unlabelled vertex, $n+2$ is not leading.

\item $\textstyle \sum_{f_8(T)}\omega(T') = y^2\partial_y\omega(T)$. 

We add an unlabelled vertex $u$ in the middle of an edge $(i,j)$ where $j$ is unlabelled, then add the vertex $n+2$ as its second child. The same reasoning as the previous case applies.

\end{enumerate}

We also needed to verify that this construction does not change the number of improper parents or leading vertices among the first $n+1$ vertices. However, it is easy to prove since for any vertex $i \in T$, $\beta_T(i) = \beta_{T'}(i)$. Moreover, in cases $4$ and $5$, we introduce a vertex $a$ in the middle of an edge $(i,j)$, so $\beta_{T'}(a) = \beta_{T'}(j)$, and the edge $(i,a)$ in $T'$ is improper if and only if $(i,j)$ was improper in $T$, while the edge $(a,j)$ is always proper. Another consequence is that any descendant of $a$ that was leading in $T$ is still leading in $T'$. 

To conclude, let $f(T) = \cup f_i(T)$, we have 
\[
\sum_{T' \in f(T)} \omega(T') = [x +nz + (y+t)(n+(y+1)\partial_y)]\omega(T).
\]
Since it is clear that $\{f(T), T \in \mathcal(G)_{n+1}\}$ is a partition of $\mathcal{G}_{n+2}$, we only need to sum the above equation over  $T \in \mathcal{G}_{n+1}$ to obtain (\ref{eq:prn}).

\end{proof}

As pointed out in the previous section, the sequence of polynomials $R_n(0,y,0,0)$ have positive coefficients, and gives the triangle of Ward numbers \seqnum{A134991}. It first means that $Q_n(0,y+1,0,-1)$ has positive coefficients, which was not immediate from Theorem \ref{th:qn}. This sequence can be interpreted as the face numbers of the space of \emph{phylogenetic trees}, introduced by Billera, Holmes and Vogtmann in \cite{billera2001geometry}, which is also the tropical Grassmannian of lines $\mathcal{G}_{2,n}$ (see \cite{speyer2004tropical}). The faces of the space of phylogenetic trees can be seen as unrooted Greg trees which only labelled vertices are its leaves. The degree of each face is the number of unlabelled vertices of the Greg tree. Remark that maximal faces are trivalent trees, of which there are $(2n-3)!!$. This definition, along with Theorem \ref{th:main1}, is enough to prove our claim. Indeed, $R_n(0,y,0,0) = \sum_{T} y^{\unl(T)}$, where $deg_1(T) = 1$, $\impp(T) = 0$, and $lead(T) = 2$. Recall that $1$ and $2$ are always leading vertices in $T$. Now assume that $T$ has at least one internal labelled vertex, and let $j \neq 1$ be one. Let us assume that $j$ is not an improper parent. Let $k$ be the smallest strict descendant of $j$, then we can easily see that $k$ is leading. Hence $T$ cannot have internal labelled vertices, which proves our claim.

\section{A bijection for Cayley Trees}
We now give a new interpretation of $R_n$, with Cayley trees this time. The proof builds on the previous theorem, and gives a statistics preserving bijection between Greg trees ans Cayley trees. This is a generalization of the second equality of Proposition~\ref{prop:h}.
\begin{theorem}
\label{th:main2}
\begin{equation}
\label{eq:main2}
R_n(x,y,z,t) = \sum_{T \in C_{n+1}} (y+1)^{\impe(T)-\impp(T)}(y+t)^{\impp(T)} x^{\deg_T(1)-1} z^{\lead(T)-\deg_T(1)-1}.
\end{equation}
\end{theorem}

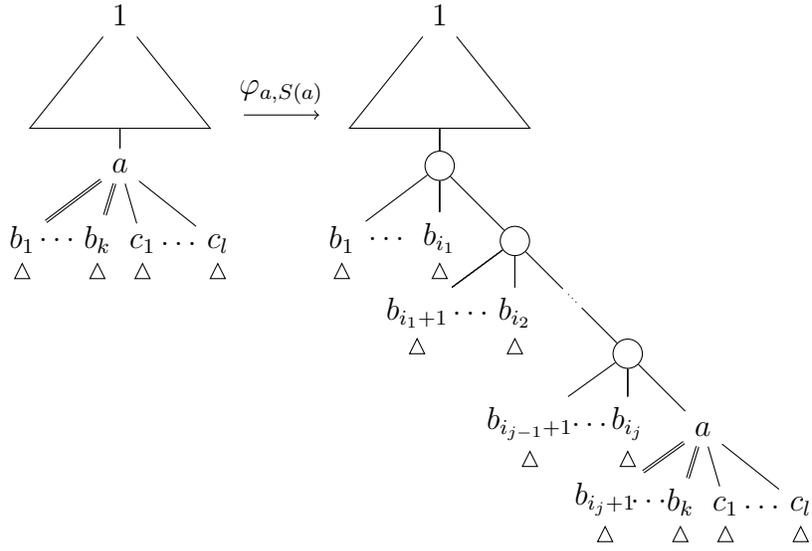
\begin{figure}
\center
\begin{tikzpicture}
\path
(0,0) node[anchor=north](a) {\tikzstyle{doto} = [dash pattern= on 0.7cm off 0.1cm on \pgflinewidth off 4\pgflinewidth on \pgflinewidth off 4\pgflinewidth on \pgflinewidth off 0.1cm on 0.7cm]

\newcommand{\tri}[1]{\draw (#1.south) -- ++(0.1,-0.2) -- ++(-0.2,0) -- cycle}
\begin{tikzpicture}[anchor=center]

\path
(0,0) node(root) {$1$}
++(0,-2) node(a) {$a$}

++(-1.3,-1.3) node[anchor=south](b1) {$b_1$}
++(1,0) node[anchor=south](bk) {$b_k$}
++(0.6,0) node[anchor=south](c1) {$c_1$}
++(1,0) node[anchor=south](cl) {$c_l$}
;

\draw (root) -- ++(1.2,-1.5) -| (a) |- ++(-1.2,0.5) -- (root);
\draw[double] (b1) -- (a) -- (bk);
\draw (c1) -- (a) -- (cl);
\draw[draw=none] (b1) -- (bk) node[midway]{\ldots};
\draw[draw=none] (c1) -- (cl) node[midway]{\ldots};
\tri{b1};
\tri{bk};
\tri{c1};
\tri{cl};
\end{tikzpicture}}
+(2.3,-1) node[anchor=north](b) {

\begin{tikzpicture}
\draw[->] (0,0) -- (1,0) node[midway,above]{$\varphi_{a,S(a)}$};
\end{tikzpicture}
}
+(6,0) node[anchor=north](c) {\tikzstyle{doto} = [dash pattern= on 0.7cm off 0.1cm on \pgflinewidth off 4\pgflinewidth on \pgflinewidth off 4\pgflinewidth on \pgflinewidth off 0.1cm on 0.7cm]

\newcommand{\tri}[1]{\draw (#1.south) -- ++(0.1,-0.2) -- ++(-0.2,0) -- cycle}
\begin{tikzpicture}[anchor=center]

\path
(0,0) node(root) {$1$}
++(0,-2) node[circle,draw](u1) {}
+(-1.3,-1.3) node[anchor=south](b11) {$b_1$}
+(0,-1.3) node[anchor=south](b12) {$b_{i_1}$}

++(1,-1) node[circle,draw](u2) {}
+(-1.3,-1.3) node[anchor=south](b21) {$b_{i_1+1}$}
+(0,-1.3) node[anchor=south](b22) {$b_{i_2}$}

++(1.5,-1.5) node[circle,draw](u3) {}
+(-1.3,-1.3) node[anchor=south](b31) {$b_{i_{j-1}+1}$}
+(0,-1.3) node[anchor=south](b32) {$b_{i_j}$}

++(1,-1) node(a) {$a$}
++(-1.3,-1.3) node[anchor=south](b1) {$b_{i_j+1}$}
++(1,0) node[anchor=south](bk) {$b_k$}
++(0.6,0) node[anchor=south](c1) {$c_1$}
++(1,0) node[anchor=south](cl) {$c_l$}
;

\draw (root) -- ++(1.2,-1.5) -| (u1) |- ++(-1.2,0.5) -- (root);
\draw (b11) -- (u1) -- (b12) -- (u1) --
(u2) -- (b21) -- (u2) -- (b22) -- (u2);
\draw[doto] (u2) -- (u3);
\draw (b31) -- (u3) -- (b32) -- (u3) -- (a);

\draw[draw=none] (b11) -- (b12) node[midway]{\ldots};
\draw[draw=none] (b21) -- (b22) node[midway]{\ldots};
\draw[draw=none] (b31) -- (b32) node[midway]{\ldots};

\draw[double] (b1) -- (a) -- (bk);
\draw (c1) -- (a) -- (cl);
\draw[draw=none] (b1) -- (bk) node[midway]{\ldots};
\draw[draw=none] (c1) -- (cl) node[midway]{\ldots};

\tri{b11};
\tri{b12};
\tri{b21};
\tri{b22};
\tri{b31};
\tri{b32};
\tri{b1};
\tri{bk};
\tri{c1};
\tri{cl};

\end{tikzpicture}};

\end{tikzpicture}
\caption{Illustration of the operation $\varphi_{a,S(a)}$. Children are shown ordered with increasing values of $\beta_T$.}
\label{img:phi}
\end{figure}
\begin{proof}
To prove (\ref{eq:main2}), we will use Theorem \ref{th:main1}, and find an application between Cayley trees and Greg trees that behaves nicely with our statistics. We first need to have some definitions. 

Let $T \in \mathcal{G}_{n+1}$, for a vertex $a \in T$, we denote by $I(a) = \{b_1, b_2, \ldots, b_k\}$ its set of improper children, ordered so that $\beta_T(b_1)~<~\beta_T(b_2)~<~\ldots~<~\beta_T(b_k)$. Let $m(a) \in I(a)$ be so that $\beta_T(m(a)) = \max_{b\in I(a)}\{\beta_T(b)\}$. We define a \emph{selection function} $S$ over the improper parents of $T$ so that, for $a \in \impp(T)$, $S(a) \subset I(a)$. Let $\Sel(T)$ be the set of such functions. Choosing a selection function for a tree $T$ is exactly like choosing a subset of the improper edges of $T$, hence $\Sel(T) \cong \mathcal{P}(impe(T))$.

We construct a bijection between Greg trees and the set $\{(T,S)|S\in \Sel(T)\}$ of Cayley trees with pointed improper edges. Let $a \in \impp(T)$ and $S(a) = \{b_{i_1},\ldots,b_{i_j}\} $ with $i_1 < \ldots < i_j$. The idea of the bijection is to divide the set of improper children for every improper parent of $T$, creating an unlabelled vertex for every subdivision we create. We define the transformation $\varphi_{a,S(a)}(T)$ as follows: in $T$, add a chain of $|S(a)|$ unlabelled vertices as direct antecedents of $a$, and for every child $c$ of $a$, if $\beta_T(b_{i_{m-1}+1}) \leq \beta_T(c) \leq \beta_T(b_{i_m})$, then remove the edge $(a,c)$ and add an edge between $c$ and the $m$-th unlabelled vertex, as illustrated in Figure~\ref{img:phi}. It is important to note that if $m(a) \in S(a)$, then $a$ is no longer an improper parent in the resulting tree. Also note that if $a_1,a_2 \in \impp(T)$, then $\varphi_{a_1,S(a_1)} \circ \varphi_{a_2,S(a_2)} = \varphi_{a_2,S(a_2)} \circ \varphi_{a_1,S(a_1)}$ for any $S$. Finally, let $\varphi(S,T) = \left( \underset{a \in \impp(T)}{\circ} \varphi_{a,S(a)} \right) (T)$. We have
\begin{equation*}
\unl(\varphi(S,T)) = \sum_{a \in \impp(T)}|S(a)|, \text{ and}
\end{equation*}\begin{equation*}
\impp(\varphi(S,T)) = \#\{a | m(a) \not\in S(a)\}.
\end{equation*}
Also note that the degree of $1$ and the number of leading vertices is unchanged by the transformation. So for $T \in \mathcal{C}_n$, we have
\begin{equation*}
\sum_{S \in \Sel(T)}y^{\unl(\varphi(S,T))}t^{\impp(\varphi(S,T))} = (y+1)^{\impe(T)-\impp(T)}(y+t)^{\impp(T)}.
\end{equation*}

This operation is reversible, for every Greg tree we can find the only Cayley tree from which it can be obtained. Formally, we introduce a rewriting system $\rightarrow$ over Greg trees. Let $G \in \mathcal{G}_{n+1}$, let $i \in G$ be an unlabelled vertex, and let $j$ be its only child that verifies $\beta_G(i) = \beta_G(j)$. Let $G'$ be the tree we obtain by merging $i$ and $j$. Then $G \rightarrow G'$. For example, if $S \in \Sel(T)$, and $S'$ is obtained from removing one element from $S$, we have $\varphi(S,T) \rightarrow \varphi(S',T)$. 

This rewriting system is locally confluent, since rewriting an unlabelled vertex is a local operation within the tree and does not impact the rest of the tree, and this system is terminating, since it removes an unlabelled vertex each step. By Newman's Lemma (see for instance \cite{reddy1990term}), this system is confluent, hence we can define $\psi(G)$ as the only minimal tree of $G$ in this system, verifying $G \rightarrow^* \psi(G)$ and $\psi(G) \in \mathcal{C}_{n+1}$. 

Another approach to finding the inverse of $\varphi$ would have been to invert each $\varphi_{a,S(a)}$. The idea here is that all unlabelled vertices introduced by the transformation are such that their child with the largest value for $\beta_T$ is an ancestor to $a$, are the only ones that verify such property, and form a path in the tree. This can be easily seen in Figure \ref{img:phi}, since we ordered children with increasing values of $\beta_T$. So instead of inverting for each unlabelled vertex, we would do so for each labelled vertex that is the highest $\beta_T$-valued child of an unlabelled vertex.

\begin{lemma}
\label{lem:varphi}
Let $T,T' \in \mathcal{C}_n$, $S \in \Sel(T)$ and $S' \in \Sel(T')$. We have
\begin{equation}
\label{eq:injphi}
\varphi(S,T) = \varphi(S',T') \Rightarrow T=T' \text{ and } S=S'.
\end{equation}
\end{lemma}
\begin{proof}
First, remark that by construction, $\psi(\varphi(S,T)) = T$ for any $S \in \Sel(T)$, which proves the first implication of (\ref{eq:injphi}). For the second part, we want to prove that we cannot make the same Greg tree from a Cayley tree and two different selection functions. At first glance, when looking at a chain of unlabelled vertices, one cannot say if it comes from one or multiple $\varphi_{a,S(a)}$. However, when looking at each improper vertex of $T$ one at a time, we see that there is only one way to revert $\varphi_{a,S(a)}$. The idea is then to build a recursive proof on the size of $S$. Let $T \in \mathcal{C}_n$ and $S,S' \in \Sel(T)$. For $I \subset \impp(T)$, we define $S_{|I}$ as the restriction of $S$ over $I$, which means, for $a \in \impp(T)$,
\[
S_{|I}(a) = \left\lbrace 
\begin{array}{l}
S(a) \text{ if } a \in I\\
\emptyset \text{ otherwise.}
\end{array}
\right.
\]
We prove by induction over the size of $I$ that
\begin{equation}
\label{eq:rec}
\varphi(S_{|I},T) = \varphi(S'_{|I},T) \Rightarrow S_{|I} = S'_{|I}.
\end{equation}
If $I = \{a\}$, we have $\varphi(S_{|\{a\}},T)= \varphi_{a,S(a)}(T)$. However, $J \mapsto \varphi_{a,J}(T)$ is injective, so $\varphi(S_{|\{a\}},T) = \varphi(S'_{|\{a\}},T) \Rightarrow S(a) = S'(a) \Rightarrow S_{|\{a\}} = S'_{|\{a\}}$.
Let $n \geq 0$ and assume that (\ref{eq:rec}) is true for any $I \subset \impp(T)$, with $|I| \leq n$. Let $J = \{a\} \cup I$ with $|J| = n+1$ and $a \not\in I$. Then we have
\[
\varphi(S_{|J},T) = \varphi_{a,S(a)}(\varphi(S_{|I},T)) = \varphi(S_{|\{a\}},\varphi(S_{|I},T)).
\]
Hence
\[
\begin{array}{r l}

&\varphi(S_{|J},T) = \varphi(S'_{|J},T) \\
\Rightarrow & \varphi(S_{|\{a\}},\varphi(S_{|I},T)) = \varphi(S'_{|\{a\}},\varphi(S'_{|I},T)) \\
\Rightarrow & S_{|\{a\}} = S'_{|\{a\}} \text{ and } \varphi(S_{|I},T) = \varphi(S'_{|I},T)\\
\Rightarrow & S_{|\{a\}} = S'_{|\{a\}} \text{ and } S_{|I} = S'_{|I} \\
\Rightarrow & S_{|J} = S'_{|J}.

\end{array}
\]
We conclude by taking $I = \impp(T)$ in (\ref{eq:rec}).
\end{proof}
Hence with Lemma \ref{lem:varphi}, $\{\{\varphi(S,T) , S \in \Sel(T) \} ,  T \in \mathcal{C}_{n+1}  \}$ is a partition of $\mathcal{G}_{n+1}$. So we have
\begin{equation*}
\begin{array}{c c c}

\sum_{G \in \mathcal{G}_{n+1}}y^{\unl(G)}t^{\impp(G)}&=&\sum_{T \in \mathcal{C}_{n+1}}\sum_{S \in \Sel(T)}y^{\unl(\varphi(S,T))}t^{\impp(\varphi(S,T))}\\
&=&\sum_{T \in \mathcal{C}_{n+1}}(y+1)^{\impe(T)-\impp(T)}(y+t)^{\impp(T)}.

\end{array}
\end{equation*}
Finally, $\varphi$ changes neither the degree of $1$, nor the number of leading vertices, which can easily be verified. %
Hence we can add the other two variables $x$ and $z$ to the equation, and conclude.
\end{proof}

\section{Between Cayley trees and planar trees}
Using a variable change and working with $R_n$, we were able to make sense of the multivariate Ramanujan polynomials in terms of Greg trees and Cayley trees. We now need to link our results back to the original observation from Guo and Zeng on the polynomials $Q_n$. Using definition \ref{def:rn}, we have a relation between labelled trees and planar labelled trees. In particular, we have the following equation.
\begin{proposition}
\label{prop:capla}
Let $n >1$. We have
\begin{equation}\label{eq:rq1}
\begin{array}{l}
\sum_{T \in \mathcal{C}_{n+1}} y^{\impe(T)} (t+1)^{\impp(T)} x^{\deg_T(1)-1} z^{\lead(T)-\deg_T(1)-1}\\
\hspace{40pt}= \sum_{T \in \mathcal{O}_{n+1}} y^{\impe(T)+\eld(T)} t^{\eld(T)} x^{\young_T(1)-1} z^{n-\impe(T) - \young_T(1) - \eld(T)}.\\
\end{array}
\end{equation}
\end{proposition}
Indeed, with a change of variables, the definition of $R_n$ gives us \[ R_n(x,y-1,z,yt+1)=Q_n(x,y,z,yt).\] Equality (\ref{eq:rq1}) is an immediate corollary from the definition of $R_n$, and the Theorems \ref{th:qn} and \ref{th:main2}. However, its summatory nature calls for a bijective proof. First, let us use Lemma \ref{lem:leadimp} and shift the indexes to obtain the following equivalent formula
\begin{equation*}
\begin{array}{l}
\sum_{T \in \mathcal{C}_{n}} y^{\impe(T)} (t+1)^{\impp(T)} x^{\deg_T(1)-1} z^{n-\impe(T)-\deg_T(1)}\\
\hspace{40pt}= \sum_{T \in \mathcal{O}_{n}} y^{\impe(T)+\eld(T)} t^{\eld(T)} x^{\young_T(1)-1} z^{n-\impe(T)  - \eld(T) - \young_T(1)}.\\
\end{array}
\end{equation*}

\begin{figure}
\center
\begin{tikzpicture}
\path
(0,0) node[anchor=north](a) {\newcommand{\triz}[1]{\draw (#1.south) -- ++(0.2,-0.4) -- ++(-0.4,0) -- cycle}
\begin{tikzpicture}[anchor=center]

\path
(0,0) node(root) {$1$}
++(0,-2) node(j) {$j$}
+(-1,-1) node(jl) {}
+(1,-1) node(jr) {}
++(0,-1) node(i) {$i$}
++(-1.3,-1.3) node[anchor=south](e1) {$e_1$}
++(1,0) node[anchor=south](ep) {$e_p$}
++(0.6,0) node[anchor=south](k) {$k$}
(i)+(0.775,-0.8) node[scale=0.75](ti) {$T_i$}
;

\draw (root) -- ++(1.2,-1.5) -| (j) |- ++(-1.2,0.5) -- (root);
\draw (jl) -- (j) -- (jr) (j) -- (i) (i) -- ++(0.6,-1) -- ++(0.8,0) -- (i);
\draw[thick,red] (ep) -- (i) -- (e1);
\draw[thick,red,double] (i) -- (k);
\draw[dotted] (e1.north west) -- (k.north east) -- (k.south east) -- (e1.south west) -- cycle;

\draw[draw=none] (jl) -- (i) node[midway]{\ldots};
\draw[draw=none] (jr) -- (i) node[midway]{\ldots};
\draw[draw=none] (e1) -- (ep) node[midway]{\ldots};

\triz{e1};
\triz{ep};
\triz{k};

\end{tikzpicture}}
+(2.5,-0.7) node[anchor=north](b) {

\begin{tikzpicture}
\draw[->] (0,0) -- (1,0) node[midway,above]{$\zeta_i$};
\draw[->] (1,-0.1) -- (0,-0.1) node[midway,below] {$\zeta_i^{-1}$};
\end{tikzpicture}
}
+(5,0) node[anchor=north](c) {\newcommand{\triz}[1]{\draw (#1.south) -- ++(0.2,-0.4) -- ++(-0.4,0) -- cycle}
\newcommand{\trizz}[1]{\draw (#1.south) -- ++(0.28,-0.56) -- ++(-0.56,0) -- cycle}
\begin{tikzpicture}[anchor=center]

\path
(0,0) node(root) {$1$}
++(0,-2) node(j) {$j$}
+(-2.1,-1) node(jl) {}
+(2.1,-1) node(jr) {}
+(-1.1,-1.3) node[anchor=south](i) {$i$}
+(-0.5,-1.3) node[anchor=south](e1) {$e_1$}
+(0.5,-1.3) node[anchor=south](ep) {$e_p$}
+(1.1,-1.3) node[anchor=south](k) {$k$}
(i)+(0.02,-0.66) node[scale=0.8](ti) {$T_i$}
;

\draw (root) -- ++(1.2,-1.5) -| (j) |- ++(-1.2,0.5) -- (root);
\draw (jl) -- (j) -- (jr) (j) -- (i);
\draw[thick,red] (j) -- (e1) (ep) -- (j) -- (k);

\draw[draw=none] (jl) -- (i) node[midway]{\ldots};
\draw[draw=none] (jr) -- (k) node[midway]{\ldots};
\draw[draw=none] (e1) -- (ep) node[midway]{\ldots};
\draw[dotted] (e1.north west) -- (k.north east) -- (k.south east) -- (e1.south west) -- cycle;

\triz{e1};
\triz{ep};
\triz{k};
\trizz{i};

\end{tikzpicture}};

\end{tikzpicture}
\caption{Illustration of the bijective operation $\zeta_i$. The vertices from $e_1$ to $k$ are moved as a singular block from being leftmost children to right siblings of $i$, \textit{et vice versa}.}
\label{img:zeta}
\end{figure}
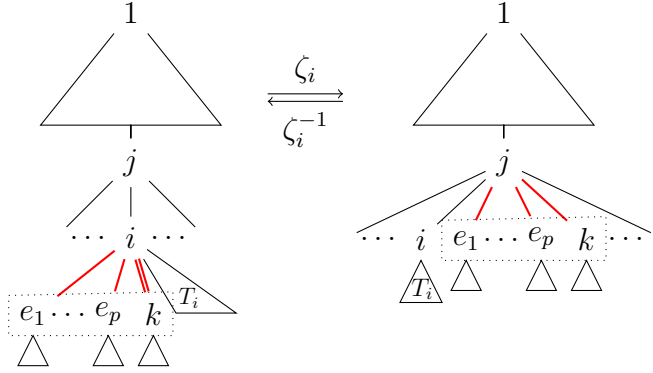

Let $T \in \mathcal{O}_n$. Let $i$ be a young vertex and improper parent of $T$. Since $1$ is a proper parent, $i$ cannot be labelled $1$, and we can introduce $j$ as the parent of $i$.

\begin{definition}
Let $\zeta_i(T)$ be the following transformation: if $k$ is the child of $i$ that verifies $\beta_T(i) = \beta_T(k)$, then we take $k$, its left siblings and all their respective subtrees, detach them from $i$ and attach them to $j$ in the same order directly on the right of $i$.
\end{definition} 
See Figure~\ref{img:zeta} for an illustration of this definition. Note that $k$ is the first younger child of $i$, so all of its left siblings are elders. Since $i$ was young, all its left siblings have a greater value for $\beta_T$ than $k$, so they remain elders. Moreover, $\beta_{zeta_i(T)}(i) > \beta_{zeta_i(T)}(k)$, so $i$ becomes an elder. Since $i$ was young, $k$ is still young in $\zeta_i(T)$. The rest of the tree remains unchanged. Hence we have increased by one the number of elders. 

However, $(i,k)$ was an improper edge, which has been removed. There are two cases to consider, first if $(j,i)$ was proper, itself and the other newly created edges are proper. Otherwise, if $(j,i)$ was improper, it is now proper since $i$ is now elder, but $(j,k)$ is improper in its stead, since $\beta_T(i) = \beta_{\zeta_i(T)}(k)$. In any case, the number of improper edges decreases by exactly one. So this transformation keeps the sum $\impe+\eld$ constant, and creates exactly one elder. It does not change the number of younger children of $1$. Finally, this operation is easily reversible: for an elder vertex $i$ in $T$, take all its right siblings up to the first younger, and move them as its leftmost children.

Note that if $i$ and $j$ are two improper parents of $T$, then $\zeta_i \circ \zeta_j (T) = \zeta_j \circ \zeta_i(T)$. Hence we can define, for $S \subset \impp(T) \cap \young(T)$, 
\begin{equation*}
\zeta(S,T) = \left( \underset{a \in S}{\circ} \zeta_a \right) (T).
\end{equation*}
Similarly, for $S \subset \eld(T)$, 
\begin{equation*}
\zeta^{-1}(S,T) = \left( \underset{a \in S}{\circ} \zeta^{-1}_a \right) (T).
\end{equation*}
\begin{definition}
For $T \in \mathcal{C}_n$, let the \emph{canonical planar tree} of $T$ $\repr(T) \in \mathcal{O}_n$ be the only planar tree whose underlying Cayley tree is $T$ with only younger vertices. In other words, its siblings are ordered from left to right with increasing values for $\beta_T$.
\end{definition}
Note that $\deg_T(1) = \young_{\repr(T)}(1)$. Let $Z(T) = \{\zeta(S,\repr(T)),S \subset \impp(\repr(T))\}$. We can now properly prove the proposition presented at the beginning of this section.

\begin{proof}[Proof of Proposition \ref{prop:capla}]
By construction of $\zeta$ as a bijection between $Z(T)$ and $\{(S,T),~S~\subset~\impp(T)\} $, for all $T \in \mathcal{C}_n$, we have
\begin{equation}
\label{eq:zetacn}
\begin{array}{l}
y^{\impe(\repr(T))} (t+1)^{\impp(\repr(T))} x^{\young_{\repr(T)}(1)-1} z^{-\impe(\repr(T))-\young_{\repr(T)}(1)}\\
\hspace{40pt}= \sum_{U \in Z(T)} y^{\impe(U)+\eld(U)} t^{\eld(U)} x^{\young_{U}(1)-1} z^{-\impe(U)  - \eld(U) - \young_{U}(1)}.\\
\end{array}
\end{equation}
Moreover, for all $U \in \mathcal{O}_n$, we have
\[
\exists T \in \mathcal{C}_n, \zeta^{-1}(\eld(U),U) = \repr(T).
\]
The idea is to apply $\zeta^{-1}$ to every elder vertex of $U$, so that the resulting tree only has younger vertices, which means it is a canonical planar tree for some Cayley tree. This also implies that $T$ is the only Cayley tree such that $U \in Z(T)$. Hence, $\{Z(T), T \in \mathcal{C}_n\}$ is a partition of $\mathcal{O}_n$. We conclude by summing (\ref{eq:zetacn}) over $\mathcal{C}_n$.

\end{proof}

\section{Acknowledgements}
We thank Matthieu Josuat-Verg\`es for his invaluable help and support. This research did not receive any specific grant from funding agencies in the public, commercial, or not-for-profit sectors.

\bibliographystyle{abbrv}
\bibliography{greg}

\begin{thebibliography}{10}

\bibitem{berndt1985ramanujan}
B.~Berndt.
\newblock Ramanujan’s second notebooks, part i, chap. 3: Combinatorial
  analysis and series inversions, 1985.

\bibitem{berndt1983ramanujan}
B.~C. Berndt.
\newblock of ramanujan's second notebook.
\newblock {\em Bulletin of the London Mathematical Society}, 15(4):273--320,
  1983.

\bibitem{billera2001geometry}
L.~J. Billera, S.~P. Holmes, and K.~Vogtmann.
\newblock Geometry of the space of phylogenetic trees.
\newblock {\em Advances in Applied Mathematics}, 27(4):733--767, 2001.

\bibitem{chapoton2002operades}
F.~Chapoton.
\newblock Op{\'e}rades diff{\'e}rentielles gradu{\'e}es sur les simplexes et
  les permutoedres.
\newblock {\em Bulletin de la soci{\'e}t{\'e} math{\'e}matique de France},
  130(2):233--252, 2002.

\bibitem{chen2001bijections}
W.~Y. Chen and V.~J. Guo.
\newblock Bijections behind the ramanujan polynomials.
\newblock {\em arXiv preprint math/0107024}, 2001.

\bibitem{dumont1996grammaire}
D.~Dumont and A.~Ramamonjisoa.
\newblock Grammaire de ramanujan et arbres de cayley.
\newblock {\em Electron. J. Combin}, 3(2), 1996.

\bibitem{flight1990many}
C.~Flight.
\newblock How many stemmata?
\newblock {\em Manuscripta}, 34(2):122--128, 1990.

\bibitem{guo2007generalization}
V.~J. Guo and J.~Zeng.
\newblock A generalization of the ramanujan polynomials and plane trees.
\newblock {\em Advances in Applied Mathematics}, 39(1):96--115, 2007.

\bibitem{josuat2015derivatives}
M.~Josuat-Verg{\`e}s.
\newblock Derivatives of the tree function.
\newblock {\em The Ramanujan Journal}, 38(1):1--15, 2015.

\bibitem{ramanujan2013notebooks}
S.~Ramanujan et~al.
\newblock {\em Notebooks of Srinivasa Ramanujan}, volume~2.
\newblock Springer, 2013.

\bibitem{reddy1990term}
U.~S. Reddy.
\newblock Term rewriting induction.
\newblock In {\em International Conference on Automated Deduction}, pages
  162--177. Springer, 1990.

\bibitem{shor1995new}
P.~W. Shor.
\newblock A new proof of cayley's formula for counting labeled trees.
\newblock {\em Journal of Combinatorial Theory, Series A}, 71(1):154--158,
  1995.

\bibitem{speyer2004tropical}
D.~Speyer and B.~Sturmfels.
\newblock The tropical grassmannian.
\newblock {\em Advances in Geometry}, 4(3):389--411, 2004.

\bibitem{zeng1999ramanujan}
J.~Zeng.
\newblock A ramanujan sequence that refines the cayley formula for trees.
\newblock {\em The Ramanujan Journal}, 3(1):45--54, 1999.

\end{thebibliography}

\end{document}